\documentclass[10pt,a4paper]{article}

\usepackage{latexsym,amsfonts,amsmath,amssymb,mathrsfs,url,color,amsthm}
\usepackage{graphicx}

\newtheorem{theorem}{Theorem}
\newtheorem{lemma}[theorem]{Lemma}

\newtheorem{remark}[theorem]{Remark}
\newtheorem{proposition}[theorem]{Proposition}
\newtheorem{definition}[theorem]{Definition}

\newcommand{\rd}{\, \mathrm{d}}
\newcommand{\bsa}{\boldsymbol{a}}
\newcommand{\bsc}{\boldsymbol{c}}
\newcommand{\bse}{\boldsymbol{e}}
\newcommand{\bsk}{\boldsymbol{k}}
\newcommand{\bsl}{\boldsymbol{l}}
\newcommand{\bst}{\boldsymbol{t}}
\newcommand{\bsw}{\boldsymbol{w}}
\newcommand{\bsx}{\boldsymbol{x}}
\newcommand{\bsy}{\boldsymbol{y}}
\newcommand{\bsz}{\boldsymbol{z}}
\newcommand{\bszero}{\boldsymbol{0}}
\newcommand{\bsnu}{\boldsymbol{\nu}}
\newcommand{\bssigma}{\boldsymbol{\sigma}}
\newcommand{\CC}{\mathbb{C}}
\newcommand{\FF}{\mathbb{F}}
\newcommand{\NN}{\mathbb{N}}
\newcommand{\RR}{\mathbb{R}}
\newcommand{\ZZ}{\mathbb{Z}}
\newcommand{\ds}{\mathrm{ds}}
\newcommand{\sym}{\mathrm{sym}}
\newcommand{\wal}{\mathrm{wal}}
\newcommand{\Ecal}{\mathcal{E}}
\newcommand{\Hcal}{\mathcal{H}}
\newcommand{\Pcal}{\mathcal{P}}
\newcommand{\Qcal}{\mathcal{Q}}
\newcommand{\Rcal}{\mathcal{R}}

\allowdisplaybreaks

\begin{document}

\title{The $b$-adic symmetrization of digital nets for quasi-Monte Carlo integration\thanks{This work was supported by JSPS Grant-in-Aid for Young Scientists No.15K20964.}}

\author{Takashi Goda\thanks{Graduate School of Engineering, The University of Tokyo, 7-3-1 Hongo, Bunkyo-ku, Tokyo 113-8656, Japan (\tt{goda@frcer.t.u-tokyo.ac.jp})}}

\date{\today}

\maketitle

\begin{abstract}
The notion of symmetrization, also known as Davenport's reflection principle, is well known in the area of the discrepancy theory and quasi-Monte Carlo (QMC) integration. In this paper we consider applying a symmetrization technique to a certain class of QMC point sets called digital nets over $\ZZ_b$. Although symmetrization has been recognized as a geometric technique in the multi-dimensional unit cube, we give another look at symmetrization as a geometric technique in a compact totally disconnected abelian group with dyadic arithmetic operations. Based on this observation we generalize the notion of symmetrization from base 2 to an arbitrary base $b\in \NN$, $b\ge 2$. Subsequently, we study the QMC integration error of symmetrized digital nets over $\ZZ_b$ in a reproducing kernel Hilbert space. The result can be applied to component-by-component construction or Korobov construction for finding good symmetrized (higher order) polynomial lattice rules which achieve high order convergence of the integration error for smooth integrands at the expense of an exponential growth of the number of points with the dimension. Moreover, we consider two-dimensional symmetrized Hammersley point sets in prime base $b$, and prove that the minimum Dick weight is large enough to achieve the best possible order of $L_p$ discrepancy for all $1\le p< \infty$.
\end{abstract}

\section{Introduction}\label{sec:intro}
In this paper we study multivariate integration of functions defined over the $s$-dimensional unit cube. For an integrable function $f:[0,1]^s\to \RR$, we denote the true integral of $f$ by
  \begin{align*}
    I(f) = \int_{[0,1]^s}f(\bsx) \rd \bsx.
  \end{align*}
We consider an approximation of $I(f)$ given in the form
  \begin{align*}
    I(f;P) = \frac{1}{|P|}\sum_{\bsx\in P}f(\bsx),
  \end{align*}
where $P\subset [0,1]^s$ denotes a finite point set in which we count points according to their multiplicity. $I(f; P)$ is called a quasi-Monte Carlo (QMC) integration rule of $f$ over $P$. Obviously the absolute error $|I(f;P)-I(f)|$ depends only on a point set $P$ for a given $f$. There are two prominent classes of point sets: digital nets \cite{DPbook,Niebook} and integration lattices \cite{Niebook,SJbook}. We are concerned with digital nets in this paper.

The quality of a point set has been often measured by the so-called discrepancy \cite{Mbook,Niebook}. For $\bst=(t_1,\ldots,t_s)\in [0,1]^s$, we denote by $[\bszero,\bst)$ the anchored axis-parallel rectangle $[0,t_1)\times [0,t_2)\times \cdots \times [0,t_s)$. The local discrepancy function $\Delta(\cdot;P):[0,1]^s \to \RR$ is defined by
  \begin{align*}
    \Delta(\bst;P) = \frac{1}{|P|}\sum_{\bsx\in P}1_{[\bszero,\bst)}(\bsx)- \prod_{j=1}^{s}t_j ,
  \end{align*}
where $1_{[\bszero,\bst)}$ denotes the characteristic function of $[\bszero,\bst)$. Now for $1\le p\le \infty$ the $L_p$ discrepancy of $P$ is defined as the $L_p$-norm of $\Delta(\cdot;P)$, i.e.,
  \begin{align*}
    L_p(P) := \left( \int_{[0,1]^s}|\Delta(\bst;P)|^p \rd \bst\right)^{1/p} ,
  \end{align*}
with the obvious modification for $p=\infty$. When $f$ has bounded variation $V(f)$ on $[0,1]^s$ in the sense of Hardy and Krause, the absolute error is bounded by
  \begin{align*}
    |I(f;P)-I(f)| \le V(f)L_{\infty}(P).
  \end{align*}
This inequality is called the Koksma-Hlawka inequality \cite[Chapter~2]{Niebook}. An inequality of similar type also holds for $1\le p< \infty$, see for instance \cite{SW98}. This is why we consider the $L_p$ discrepancy as a quality measure of a point set.

It is, however, not an easy task to construct point sets with low $L_p$ discrepancy. As an example, let us consider the two-dimensional Hammersley point sets in base $b$ defined as follows.
\begin{definition}\label{def:hammersley}
Let $b\ge 2$ be a positive integer. For $m\in \NN$, the two-dimensional Hammersley point set in base $b$ consisting of $b^m$ points is defined as
  \begin{align*}
    P_H := \left\{ \left( \frac{a_1}{b}+\cdots + \frac{a_m}{b^m}, \frac{a_m}{b}+\cdots + \frac{a_1}{b^m}\right): a_i\in \{0,1,\ldots,b-1\}\right\}.
  \end{align*}
\end{definition}
\noindent It is known that $P_H$ has optimal order of the $L_{\infty}$ discrepancy, while it does not have optimal order of the $L_p$ discrepancy for all $p\in [1,\infty)$, see for instance \cite{Pill02}. Here the general lower bounds on the $L_p$ discrepancy for all $s\in \NN$ have been shown by Roth \cite{Rot54} for $p\ge 2$, Schmidt \cite{Sch72,Sch77} for $p>1$ and Hal\'{a}sz \cite{Hal81} for $p=1$. Note that the optimal orders of the $L_1$ and $L_{\infty}$ discrepancy are still unknown for $s\ge 3$.

There are several ways to modify $P_H$ such that the modified point set has optimal order of $L_p$ discrepancy for $p\in [1,\infty)$, see for instance \cite{Bil11,Goda14,HKP15}. The notion of symmetrization (also known as Davenport's reflection principle) \cite{Dav56} is one of the best-known remedies in order to achieve optimal order of $L_p$ discrepancy and has been thoroughly studied in the literature, see for instance \cite{HKP15,Krixx,LP01,LP02,Pro88}. We also refer to \cite{DNP14} in which symmetrized point sets are studied in the context of QMC rules using integration lattices.

Although the original symmetrization due to Davenport has been recognized as a geometric technique in the $s$-dimensional unit cube, in this paper, we give another look at symmetrization as that in a compact totally disconnected abelian group with dyadic arithmetic operations. This implies that symmetrization fits quite well with dyadic structure of point sets and also with the tools used for analyzing the point sets. Therefore, it is quite reasonable to consider symmetrized two-dimensional Hammersley point sets in base 2 \cite{HKP15,LP01}, or more generally, symmetrized digital nets in base 2 \cite{LP02}. Although there are some exceptions as in \cite{Krixx,Pro88} where symmetrization is shown to be helpful even if point sets have not dyadic structure, it must be interesting to find a geometric symmetrization technique which acts on a compact totally disconnected abelian group with $b$-dic arithmetic operations for $b\ge 2$.

The aim of this paper is two-fold: to generalize the notion of symmetrization from base 2 to an arbitrary base $b\in \NN$, $b\ge 2$ and to obtain some basic results on the QMC integration error of symmetrized digital nets in base $b$. In particular, we study the worst-case error of symmetrized digital nets in base $b$ in a reproducing kernel Hilbert space (RKHS) in Section~\ref{sec:qmc}. This result for digital nets can be regarded as an analog of the result for lattice rules obtained in \cite[Section~4.2]{DNP14} where only the sum of the half-period cosine space and the Korobov space is considered as a RKHS. In Section~\ref{sec:qmc}, we also study the mean square worst-case error with respect to a random digital shift in a RKHS. Furthermore, in Section~\ref{sec:ham}, we prove that symmetrized Hammersley point sets in base $b$ achieve the best possible order of $L_p$ discrepancy for all $1\le p< \infty$.

{\bf Notation.} Let $\NN$ be the set of positive integers and $\NN_0:=\NN \cup \{0\}$. Let $\CC$ be the set of all complex numbers. For a positive integer $b\ge 2$, $\ZZ_b$ denotes the residue class ring modulo $b$, which is identified with the set $\{0,1,\ldots,b-1\}$ equipped with addition and multiplication modulo $b$. For $x\in [0,1]$, its $b$-adic expansion $x=\sum_{i=1}^{\infty}\xi_i b^{-i}$ with $\xi_i \in \ZZ_b$ is unique in the sense that infinitely many of the $\xi_i$ are different from $b-1$ except for the endpoint $x=1$ for which all $\xi_i$'s are equal to $b-1$. Note that for $1\in \NN$ we use the $b$-adic expansion $1\cdot b^0$, whereas for $1\in [0, 1]$ we use the $b$-adic expansion $(b-1)(b^{-1}+b^{-2}+\ldots)$. It will be clear from the context which expansion we use.

\section{Preliminaries}\label{sec:pre}
Here we recall necessary background and notation, including infinite direct products of $\ZZ_b$, Walsh functions and digital nets (with infinite digit expansions) over $\ZZ_b$. We essentially follow the exposition of \cite[Section~2]{GSY2}.
\subsection{Infinite direct products of $\ZZ_b$}\label{subsec:inf_pro}
Let us define $G=(\ZZ_b)^{\NN}$, which is a compact totally disconnected abelian group with the product topology where $\ZZ_b$ is considered to be a discrete group. We denote by $\oplus$ and $\ominus$ addition and subtraction in $G$, respectively. Let $\nu$ be the product measure on $G$ inherited from $\ZZ_b$, that is, for every cylinder set $E = \prod_{i=1}^n Z_i \times \prod_{i\ge n+1} \ZZ_b$ with $Z_i \subseteq \ZZ_b$ for $1\le i\le n$, we have $\nu(E) = (\prod_{i=1}^n |Z_i|)/{b^n}$.

A character on $G$ is a continuous group homomorphism from $G$ to $\{z\in \CC : |z|=1\}$, which is a multiplicative group of complex numbers whose absolute value is 1. We define the $k$-th character as follows.

\begin{definition}\label{def:char_1}
For a positive integer $b\ge 2$, let $\omega:=\exp(2\pi \sqrt{-1}/b)$ be the primitive $b$-th root of unity. Let $z = (\zeta_1, \zeta_2, \dots) \in G$ and $k \in \NN_0$ whose $b$-adic expansion is given by $k = \kappa_0+\kappa_1b+\dots+\kappa_{a-1}b^{a-1}$ with $\kappa_0,\ldots,\kappa_{a-1}\in \ZZ_b$. Then the $k$-th character $W_k: G \to \{1, \omega, \dots, \omega^{b-1}\}$ is defined as
\begin{align}
W_k(z) := \omega^{\kappa_0 \zeta_1 + \cdots + \kappa_{a-1} \zeta_a}.
\end{align}
\end{definition}
\noindent
We note that every character on $G$ is equal to some $W_k$, see \cite{Ponbook}.

Let us now consider the higher-dimensional case. Let $G^s$ denote the $s$-ary Cartesian product of $G$. Note that $G^s$ is also a compact totally disconnected abelian group with the product topology. The operators $\oplus$ and $\ominus$ now denote addition and subtraction in $G^s$, respectively. We denote by $\bsnu$ the product measure on $G^s$ inherited from $\nu$. The $\bsk$-th character on $G^s$ can be defined as follows.

\begin{definition}\label{def:char_s}
For a positive integer $b\ge 2$ and a dimension $s\in \NN$, let $\bsz=(z_1,\ldots, z_s) \in G^s$ and $\bsk=(k_1,\ldots, k_s)\in \NN_0^s$.
Then the $\bsk$-th character $W_{\bsk}: G^s \to \{1,\omega_b,\ldots, \omega_b^{b-1}\}$
is defined as
  \begin{align*}
    W_{\bsk}(\bsz) := \prod_{j=1}^s W_{k_j}(z_j) .
  \end{align*}
\end{definition}
\noindent
We note again that every character on $G^s$ is equal to some $W_{\bsk}$ as with the one-dimensional case.

The group $G$ is related to the unit interval $[0,1]$ through the following maps $\pi: G \to [0,1]$ and $\sigma: [0,1] \to G$. Let $z=(\zeta_1, \zeta_2, \dots) \in G$, and let $x \in [0,1]$ with its unique $b$-adic expansion $x=\sum_{i=1}^{\infty}\xi_i b^{-i}$ with $\xi_i \in \ZZ_b$. Then the projection map $\pi$ is defined as $\pi(z) := \sum_{i=1}^\infty \zeta_i b^{-i}$ and the section map $\sigma$ is defined as $\sigma(x) := (\xi_1, \xi_2, \dots)$. By definition, $\pi$ is surjective and $\sigma$ is injective.  For the $s$-dimensional case, both the projection and section maps are applied componentwise, through which the group $G^s$ is related to the unit cube $[0,1]^s$. We note that $\pi$ is a continuous map and that $\pi \circ \sigma=\mathrm{id}_{[0,1]^s}$. We summarize some important facts below \cite[Lemma~4, Propositions~3 and 5]{GSY2}.

\begin{proposition}\label{prop:inf_prod}
The following holds true:
\begin{enumerate}
\item For $k \in \NN_0$, we have
  \begin{align*}
    \int_{G} W_{k}(z)\rd \nu(z) = \begin{cases}
     1 & \text{if $k=0$},  \\
     0 & \text{otherwise}.
    \end{cases}
  \end{align*}
\item For $\bsk,\bsl\in \NN_0^s$, we have
  \begin{align*}
    \int_{G^s} W_{\bsk}(\bsz)\overline{W_{\bsl}(\bsz)}\rd \bsnu(\bsz) = \begin{cases}
     1 & \text{if $\bsk=\bsl$},  \\
     0 & \text{otherwise}.
    \end{cases}
  \end{align*}
\item For $f \in L^1(G^s)$, we have
  \begin{align*}
    \int_{G^s} f(\bsz) \rd \bsnu(\bsz) = \int_{[0,1]^s} f (\sigma(\bsx)) \rd \bsx.
  \end{align*}
\item For $f \in L^1([0,1]^s)$, we have
  \begin{align*}
    \int_{[0,1]^s} f(\bsx)  \rd \bsx = \int_{G^s} f(\pi(\bsz)) \rd \bsnu(\bsz).
  \end{align*}
\item Let $H_n:=\{z=(\zeta_1, \zeta_2, \dots) \in G: \zeta_1=\zeta_2=\cdots =\zeta_n=0 \}$. Then we have
  \begin{align*}
    \sum_{\substack{\bsk \in \NN_0^s\\ k_j<b^n, \forall j}} W_{\bsk}(\bsz) = \begin{cases}
     b^{sn} & \text{if $\bsz \in H_n^s$}, \\
     0 & \text{otherwise}.
\end{cases}
  \end{align*}
\end{enumerate}
\end{proposition}

\subsection{Walsh functions}\label{subsec:walsh}
Walsh functions play a central role in the analysis of digital nets. We refer to \cite[Appendix~A]{DPbook} for general information on Walsh functions in the context of numerical integration. We first give the definition for the one-dimensional case.

\begin{definition}
For a positive integer $b\ge 2$, let $\omega_b=\exp(2\pi \sqrt{-1}/b)$. We denote the $b$-adic expansion of $k\in \NN_0$ by $k = \kappa_0+\kappa_1b+\dots+\kappa_{a-1}b^{a-1}$ with $\kappa_0,\ldots,\kappa_{a-1}\in \ZZ_b$. Then the $k$-th $b$-adic Walsh function ${}_b\wal_k : [0,1]\to \{1,\omega_b,\dots,\omega_b^{b-1}\}$ is defined as
  \begin{align*}
    {}_b\wal_k(x) := \omega_b^{\kappa_0\xi_1+\dots+\kappa_{a-1}\xi_a} ,
  \end{align*}
for $x\in [0,1]$ with its unique $b$-adic expansion $x=\xi_1b^{-1}+\xi_2b^{-2}+\cdots$.
\end{definition}
\noindent
This definition can be generalized to the higher-dimensional case.

\begin{definition}
For a positive integer $b\ge 2$ and a dimension $s\in \NN$, let $\bsx=(x_1,\ldots, x_s)\in [0,1]^s$ and $\bsk=(k_1,\ldots, k_s)\in \NN_0^s$. Then the $\bsk$-th $b$-adic Walsh function ${}_b\wal_{\bsk}: [0,1]^s \to \{1,\omega_b,\ldots, \omega_b^{b-1}\}$ is defined as
  \begin{align*}
    {}_b\wal_{\bsk}(\bsx) := \prod_{j=1}^s {}_b\wal_{k_j}(x_j) .
  \end{align*}
\end{definition}
Since we shall always use Walsh functions in a fixed base $b$, we omit the subscript and simply write $\wal_k$ or $\wal_{\bsk}$ in this paper.
From the definitions of characters on $G^s$ and Walsh functions, we see that for any $\bsx\in [0,1]^s$
\begin{equation}\label{eq:WandWalsh}
\wal_{\bsk}(\bsx) = W_{\bsk}(\sigma (\bsx)).
\end{equation}
 
Since the system $\{\wal_{\bsk}: \bsk\in \NN_0^s\}$ is a complete orthonormal system in $L_2([0,1]^s)$ \cite[Theorem~A.11]{DPbook}, we have a Walsh series expansion for $f\in L_2([0,1]^s)$
  \begin{align*}
     \sum_{\bsk\in \NN_0^s}\hat{f}(\bsk)\wal_{\bsk} ,
  \end{align*}
where the $\bsk$-th Walsh coefficient is given by
  \begin{align*}
     \hat{f}(\bsk) = \int_{[0,1]^s} f(\bsx)\overline{\wal_{\bsk}(\bsx)}\rd \bsx .
  \end{align*}
We refer to \cite[Appendix~A.3]{DPbook} and \cite[Lemma~18]{GSY1} for a discussion about the pointwise absolute convergence of the Walsh series to $f$.

\subsection{Digital nets}\label{subsec:digital_net}
Here we define digital nets in $G^s$ by using infinite-column generating matrices, i.e., generating matrices whose each column can contain infinitely many non-zero entries. This definition has been recently given in \cite{GSY2}.

\begin{definition}\label{def:digital_net}
For $m,s\in \NN$, let $C_1,\ldots,C_s\in \ZZ_b^{\NN \times m}$. For $0\le n<b^m$, denote the $b$-adic expansion of $n$ by $n=\sum_{i=0}^{m-1}\nu_ib^{i}$ with $\nu_i\in \ZZ_b$. Put $\bsz_n=(z_{n,1},\ldots,z_{n,s})\in G^s$ with
  \begin{align*}
     z_{n,j} = (\nu_0,\nu_1,\ldots,\nu_{m-1}) \cdot C_{j}^{\top} ,
  \end{align*}
for $1\le j\le s$. Then the set $\Pcal=\{\bsz_0,\ldots,\bsz_{b^m-1}\}\subset G^s$ is called a digital net over $\ZZ_b$ in $G^s$ with generating matrices $C_1,\ldots,C_s$.

Moreover, the set $P:=\{\pi(\bsz)\colon \bsz\in \Pcal\}\subset [0,1]^s$ is called a digital net over $\ZZ_b$ in $[0,1]^s$ with generating matrices $C_1,\ldots,C_s$.
\end{definition}

We note that every digital net in $G^s$ is a $\ZZ_b$-module of $G^s$ as well as a subgroup of $G^s$. If each column of generating matrices consists of only finitely many non-zero entries, the above definition of a digital net over $\ZZ_b$ in $[0,1]^s$ reduces to that given by Niederreiter \cite{Niebook}.

The dual net of a digital net plays an important role in the subsequent analysis. For a digital net $\Pcal$ in $G^s$, we denote its dual net by $\Pcal^{\perp}\subset \NN_0^s$ which is defined as follows.
\begin{definition}\label{def:dual_net}
Let $\Pcal$ be a digital net in $G^s$. The dual net of $\Pcal$ is defined as
  \begin{align*}
     \Pcal^{\perp}:=\{\bsk=(k_1,\ldots,k_s)\in \NN_0^s: \vec{k}_1C_1\oplus \dots \oplus \vec{k}_sC_s= (0,\ldots,0)\in \ZZ_b^m \} ,
  \end{align*}
where we write $\vec{k}_j=(\kappa_{j,0},\kappa_{j,1},\ldots)$ for $k_j$ with its $b$-adic expansion $k_j=\kappa_{j,0}+\kappa_{j,1}b+\cdots$, which is indeed a finite expansion.
\end{definition}

Since $W_{\bsk}$'s are characters on $G^s$, the following lemma can be established from Definition~\ref{def:dual_net}, which connects a digital net in $G^s$ with characters.
\begin{lemma}\label{lem:dual_Walsh}
Let $\Pcal$ be a digital net in $G^s$ and $\Pcal^{\perp}$ be its dual net. For $\bsk\in \NN_0^s$, we have
  \begin{align*}
     \sum_{\bsz\in \Pcal} W_{\bsk}(\bsz) = \begin{cases}
     |\Pcal| & \text{if $\bsk\in \mathcal{P}^{\perp}$},  \\
     0 & \text{otherwise}.  \\
    \end{cases}
  \end{align*}
\end{lemma}

\section{The $b$-adic symmetrization}\label{sec:sym}
In this section we generalize the notion of symmetrization from base 2 to an arbitrary base $b\in \NN$, $b\ge 2$. Before that, we recall the original symmetrization introduced by Davenport \cite{Dav56} and give another look at it as a geometric technique in $G^s$ with $b=2$. Let $P\subset [0,1]^2$ be a finite point set. Then the symmetrized point set in the sense of Davenport is defined as
  \begin{align*}
     P^{\sym,D}:=\{(x,y) \cup (x,1-y): (x,y)\in P\}.
  \end{align*}
It is often the case that the symmetrized point set is defined as
  \begin{align*}
     P^{\sym}:=\{(x,y) \cup (x,1-y) \cup (1-x,y) \cup (1-x,1-y) : (x,y)\in P\}.
  \end{align*}
In the remainder of this paper we only consider the symmetrized point set defined in the latter sense. For the higher-dimensional case, the symmetrized point set of $P\subset [0,1]^s$ is defined as
  \begin{align*}
     P^{\sym}:=\left\{\sym_u(\bsx): \bsx \in P, u\subseteq \{1,\ldots,s\}\right\},
  \end{align*}
where $\sym_u(\bsx)$ denotes the $s$-dimensional vector whose $j$-th coordinate is $1-x_j$ if $j\in u$ and $x_j$ otherwise, that is, $\sym_u(\bsx)=(y_1,\ldots,y_s)$ with
  \begin{align*}
     y_j = \begin{cases}
     1-x_j & \text{if $j\in u$,} \\
     x_j & \text{otherwise.} \\
     \end{cases}
  \end{align*}
By definition, we have $|P^{\sym}|=2^s |P|$.

Here we give another look at the original symmetrization. Let $b=2$ and $z=(\zeta_1,\zeta_2,\ldots)\in G$ with $\zeta_i\in \ZZ_2$. By denoting $e=(1,1,\ldots)\in G$, we have $z\oplus e=(1-\zeta_1,1-\zeta_2,\ldots)\in G$ and thus
  \begin{align*}
     \pi(z\oplus e) & =\frac{1-\zeta_1}{2}+\frac{1-\zeta_2}{2^2}+\cdots \\
     & = \left( \frac{1}{2}+\frac{1}{2^2}+\cdots \right) -\left( \frac{\zeta_1}{2}+\frac{\zeta_2}{2^2}+\cdots \right) = 1-\pi(z).
  \end{align*}
For $\bsz\in G^s$, we denote by $\sym^G_u(\bsz)$ the $s$-dimensional vector whose $j$-th coordinate is $z_j\oplus e$ if $j\in u$ and $z_j$ otherwise. Then the symmetrized point set of $\Pcal\subset G^s$ can be given by 
  \begin{align*}
     \Pcal^{\sym}:=\left\{\sym^G_u(\bsz): \bsz \in \Pcal, u\subseteq \{1,\ldots,s\}\right\}.
  \end{align*}

As a natural extension from $b=2$ to an arbitrary positive integer $b\ge 2$, we now introduce the notion of the $b$-adic symmetrization. For $l\in \ZZ_b$, we write $e_l=(l,l,\ldots)\in G$. For a vector $\bsl=(l_1,\dots,l_s)\in \ZZ_b^s$, we write $\bse_{\bsl}=(e_{l_1},\ldots,e_{l_s})\in G^s$. 
\begin{definition}\label{def:b_sym}
For a point set $\Pcal\subset G^s$, its symmetrized point set is defined as
  \begin{align*}
     \Pcal^{\sym}:=\left\{\bsz\oplus \bse_{\bsl} : \bsz \in \Pcal, \bsl \in \ZZ_b^s\right\}.
  \end{align*}
For a point set $P\subset [0,1]^s$, its symmetrized point set is defined as
  \begin{align*}
     P^{\sym} = \pi\left[(\sigma(P))^{\sym}\right].
  \end{align*}
\end{definition}
\noindent By definition, we have $|\Pcal^{\sym}|=b^s |\Pcal|$ and $|P^{\sym}|=b^s |P|$.

\subsection{Symmetrized digital nets}\label{subsec:sym_digital_net}
In the following let $\Pcal$ be a digital net in $G^s$. We write $\Qcal=\{\bse_{\bsl} : \bsl \in \ZZ_b^s\}$, which is also a digital net in $G^s$. From Definition~\ref{def:b_sym}, the symmetrized point set $\Pcal^{\sym}$ can be regarded as the direct sum of two digital nets $\Pcal$ and $\Qcal$. Thus, it is obvious that the following holds true.

\begin{lemma}\label{lem:sym_digital_net}
For a digital net $\Pcal$ in $G^s$, let $\Pcal^{\sym}$ be the symmetrized point set of $\Pcal$. Then $\Pcal^{\sym}$ is also a digital net in $G^s$.
\end{lemma}

\begin{remark}\label{rem:sym_matrix}
Consider a digital net $\Pcal$ in $G^s$ constructed with infinite-row generating matrices $C_1,\ldots,C_s\in \ZZ_b^{\NN \times m}$. Then $\Pcal^{\sym}$ is a digital net in $G^s$ whose generating matrices $D_1,\ldots,D_s\in \ZZ_b^{\NN \times (m+s)}$ are given as
  \begin{align*}
     D_j = (C_j, E_j) \quad \text{with} \quad  E_j= \bordermatrix{
         & 1 & \cdots & j-1 & j & j+1 & \cdots & s \cr
         & 0 & \cdots & 0 & 1 & 0 & \cdots & 0 \cr
         & 0 & \cdots & 0 & 1 & 0 & \cdots & 0 \cr
         & 0 & \cdots & 0 & 1 & 0 & \cdots & 0 \cr
         & & \ddots & & \vdots & & \ddots &  \cr} \in \ZZ_b^{\NN \times s} .
  \end{align*}
\end{remark}

For a point set $\Pcal^{\sym}$, we have the following orthogonal property.
\begin{lemma}\label{lem:sym_dual_Walsh}
Let $\Pcal$ be a digital net in $G^s$ and $\Pcal^{\sym}$ be its symmetrized point set. For $\bsk\in \NN_0^s$, we have
  \begin{align*}
     \sum_{\bsz\in \Pcal^{\sym}} W_{\bsk}(\bsz) = \begin{cases}
     |\Pcal^{\sym}| & \text{if $\bsk\in \Pcal^{\perp}\cap \Ecal^s$},  \\
     0 & \text{otherwise}.  \\
    \end{cases}
  \end{align*}
In the above, $\Ecal := \{k\in \NN_0: \delta(k)\equiv 0 \pmod b\}$ where $\delta(k)$ denotes the $b$-adic sum-of-digits of $k$ and is given as $\delta(k):=\kappa_0+\kappa_1+\cdots$ for $k=\kappa_0+\kappa_1 b+\cdots$.
\end{lemma}

\begin{proof}
From the definition of the $b$-adic symmetrization, we have
  \begin{align*}
     \sum_{\bsz\in \Pcal^{\sym}} W_{\bsk}(\bsz) & = \sum_{\bsz\in \Pcal}\sum_{\bsl\in \ZZ_b^s} W_{\bsk}(\bsz\oplus \bse_{\bsl}) \\
     & = \sum_{\bsz\in \Pcal}W_{\bsk}(\bsz)\sum_{\bsl\in \ZZ_b^s} W_{\bsk}(\bse_{\bsl}) .
  \end{align*}
Since $\Pcal$ is a digital net in $G^s$, the first sum in the last expression equals $|\Pcal|$ if $\bsk\in \Pcal^{\perp}$ and 0 otherwise. On the second sum in the last expression, we have
  \begin{align*}
     \sum_{\bsl\in \ZZ_b^s} W_{\bsk}(\bse_{\bsl}) & = \sum_{\bsl\in \ZZ_b^s} \prod_{j=1}^{s}\omega_b^{l_j \delta(k_j)} = \prod_{j=1}^{s} \sum_{l_j\in \ZZ_b} \omega_b^{l_j \delta(k_j)} .
  \end{align*}
As $\omega_b$ denotes the primitive $b$-th root of unity, the inner sum on the rightmost-side above equals $b$ if $\delta(k_j)\equiv 0 \pmod b$ and 0 otherwise. Thus,
  \begin{align*}
     \sum_{\bsl\in \ZZ_b^s} W_{\bsk}(\bse_{\bsl}) = \begin{cases}
     b^s & \text{if $\bsk\in \Ecal^s$,} \\
     0 & \text{otherwise.} \\
     \end{cases}
  \end{align*}
All together, we obtain
  \begin{align*}
     \sum_{\bsz\in \Pcal^{\sym}} W_{\bsk}(\bsz) = \begin{cases}
     b^s|\Pcal| & \text{if $\bsk\in P^{\perp}\cap \Ecal^s$,} \\
     0 & \text{otherwise.} \\
     \end{cases}
  \end{align*}
Since $|\Pcal^{\sym}|=b^s |\Pcal|$, the result follows.
\end{proof}
\noindent Combining Lemmas~\ref{lem:dual_Walsh}, \ref{lem:sym_digital_net} and \ref{lem:sym_dual_Walsh} implies that
  \begin{align*}
     (\Pcal^{\sym})^{\perp}=\Pcal^{\perp}\cap \Ecal^s .
  \end{align*}

\begin{remark}
In fact, the argument in this subsection can be generalized in the following way. Let $\Pcal$ and $\Pcal'$ be digital nets in $G^s$. Consider the direct sum
  \begin{align*}
     \Rcal= \{\bsz\oplus \bsz': \bsz\in \Pcal, \bsz'\in \Pcal' \}.
  \end{align*}
Then $\Rcal$ is also a digital net in $G^s$ and satisfies the orthogonal property
  \begin{align*}
     \sum_{\bsz\in \Rcal} W_{\bsk}(\bsz) = \begin{cases}
     |\Rcal| & \text{if $\bsk\in \Pcal^{\perp}\cap (\Pcal')^{\perp}$},  \\
     0 & \text{otherwise}.  \\
    \end{cases}
  \end{align*}
In the remainder of this paper, however, we only consider the case $\Pcal'=\Qcal$, which gives us $\Rcal=\Pcal^{\sym}$.
\end{remark}

\section{QMC integration over symmetrized digital nets}\label{sec:qmc}
Let us consider a reproducing kernel Hilbert space (RKHS) $\Hcal$ with reproducing kernel $K:[0,1]^s\times [0,1]^s\to \RR$. The inner product in $\Hcal$ is denoted by $\langle f,g \rangle_{\Hcal}$ for $f,g\in \Hcal$ and the associated norm is denoted by $\lVert f\rVert_{\Hcal}:=\sqrt{\langle f,f \rangle_{\Hcal}}$. It is known that if $\int_{[0,1]^s}\sqrt{K(\bsx,\bsx)}\rd \bsx<\infty$ the squared worst-case error in the space $\Hcal$ of a QMC integration over a point set $P\subset [0,1]^s$ is given by
  \begin{align}\label{eq:worst-case_error}
     e^2(P,K) & := \left( \sup_{\substack{f\in \Hcal \\ \|f\|_{\Hcal}\le 1}}|I(f)-I(f;P)|\right)^2 \nonumber \\ 
     & = \int_{[0,1]^{2s}}K(\bsx,\bsy)\rd \bsx \rd \bsy-\frac{2}{|P|}\sum_{\bsx\in P}\int_{[0,1]^s}K(\bsx,\bsy)\rd \bsy+\frac{1}{|P|^2}\sum_{\bsx,\bsy\in P}K(\bsx,\bsy) ,
  \end{align}
see for instance \cite{SW98}. For $\bsk, \bsl \in \NN_0^s$, the $(\bsk,\bsl)$-th Walsh coefficient of $K$ is defined by
  \begin{align*}
     \hat{K}(\bsk,\bsl) := \int_{[0,1]^{2s}} K(\bsx,\bsy)\overline{\wal_{\bsk}(\bsx)}\wal_{\bsl}(\bsy)\rd \bsx \rd \bsy.
\end{align*}
In the following we always assume $\int_{[0,1]^s}\sqrt{K(\bsx,\bsx)}\rd \bsx<\infty$. We study the worst-case error of symmetrized point sets in a RKHS first and then move on to the mean square worst-case error with respect to a random digital shift.

\subsection{The worst-case error}\label{sec:worst}
From the proof of \cite[Proposition~21]{GSY2}, we have the following pointwise absolute convergence of the Walsh series of $K$.
\begin{lemma}\label{lem:kernel_walsh_conv}
Let $K$ be a continuous reproducing kernel. We assume $$\sum_{\bsk,\bsl\in \NN_0^s}|\hat{K}(\bsk,\bsl)|<\infty.$$
For any $\bsz,\bsw\in G^s$, we have
\begin{align}\label{eq:kernel_walsh_conv}
    K(\pi(\bsz), \pi(\bsw)) = \sum_{\bsk,\bsl\in \NN_0^s}\hat{K}(\bsk,\bsl)W_{\bsk}(\bsz)\overline{W_{\bsl}(\bsw)}.
\end{align}
\end{lemma}

Under some assumptions on $K$, the worst-case error is given as follows.
\begin{theorem}\label{thm:worst_case_sym_digital_net}
Let $\Pcal, \Pcal^{\perp}$ be a digital net in $G^s$ and its dual net, respectively, and $\Pcal^{\sym}$ be the symmetrized point set of $\Pcal$. Let $K$ be a continuous reproducing kernel which satisfies $$\int_{[0,1]^s}\sqrt{K(\bsx,\bsx)}\rd \bsx<\infty \quad \text{and}\quad \sum_{\bsk,\bsl\in \NN_0^s}|\hat{K}(\bsk,\bsl)|<\infty.$$ The squared worst-case error of a QMC integration over $\pi(\Pcal^{\sym})$ is given by
  \begin{align*}
     e^2(\pi(\Pcal^{\sym}), K) = \sum_{\bsk,\bsl\in \Pcal^{\perp}\cap \Ecal^s\setminus \{\bszero\}}\hat{K}(\bsk,\bsl) .
  \end{align*}
\end{theorem}
\noindent Although the proof is almost the same with that of \cite[Proposition~21]{GSY2}, we provide it below for the sake of completeness.

\begin{proof}
We evaluate the three terms of (\ref{eq:worst-case_error}) separately in which $\pi(\Pcal^{\sym})$ is substituted into $P$. The first term of (\ref{eq:worst-case_error}) is simply
  \begin{align*}
     \int_{[0,1]^{2s}}K(\bsx,\bsy)\rd \bsx \rd \bsy = \hat{K}(\bszero,\bszero) ,
  \end{align*}
by the definition of the Walsh coefficients. For the second term of (\ref{eq:worst-case_error}), we have
\begin{align*}
     & \quad \frac{2}{|\Pcal^{\sym}|}\sum_{\bsz\in \Pcal^{\sym}}\int_{[0,1]^s}K(\pi(\bsz),\bsy)\rd \bsy \\
     & = \frac{1}{|\Pcal^{\sym}|}\sum_{\bsz\in \Pcal^{\sym}}\int_{[0,1]^s}K(\pi(\bsz),\bsy)\rd \bsy +\frac{1}{|\Pcal^{\sym}|}\sum_{\bsz\in \Pcal^{\sym}}\int_{[0,1]^s}K(\bsy,\pi(\bsz))\rd \bsy \\
     & = \frac{1}{|\Pcal^{\sym}|}\sum_{\bsz\in \Pcal^{\sym}}\int_{G^s}K(\pi(\bsz),\pi(\bsw))\rd \bsnu(\bsw) +\frac{1}{|\Pcal^{\sym}|}\sum_{\bsz\in \Pcal^{\sym}}\int_{G^s}K(\pi(\bsw),\pi(\bsz))\rd \bsnu(\bsw)\\
     & = \frac{1}{|\Pcal^{\sym}|}\sum_{\bsz\in \Pcal^{\sym}}\sum_{\bsk,\bsl\in \NN_0^s}\hat{K}(\bsk,\bsl)W_{\bsk}(\bsz)\int_{G^s}\overline{W_{\bsl}(\bsw)}\rd \bsnu(\bsw) \\
     & \quad + \frac{1}{|\Pcal^{\sym}|}\sum_{\bsz\in \Pcal^{\sym}}\sum_{\bsk,\bsl\in \NN_0^s}\hat{K}(\bsk,\bsl)\overline{W_{\bsl}(\bsz)}\int_{G^s}W_{\bsk}(\bsw)\rd \bsnu(\bsw) \\
     & = \sum_{\bsk\in \NN_0^s}\hat{K}(\bsk,\bszero) \frac{1}{|\Pcal^{\sym}|}\sum_{\bsz\in \Pcal^{\sym}}W_{\bsk}(\bsz)+\sum_{\bsl\in \NN_0^s}\hat{K}(\bszero,\bsl) \overline{\frac{1}{|\Pcal^{\sym}|}\sum_{\bsz\in \Pcal^{\sym}}W_{\bsl}(\bsz)} \\
     & = \sum_{\bsk\in \Pcal^{\perp}\cap \Ecal^s}\hat{K}(\bsk,\bszero)+\sum_{\bsl\in \Pcal^{\perp}\cap \Ecal^s}\hat{K}(\bszero,\bsl) ,
  \end{align*}
where we use the symmetry of $K$, Item~4 of Proposition~\ref{prop:inf_prod}, Equation~(\ref{eq:kernel_walsh_conv}), Item~1 of Proposition~\ref{prop:inf_prod} and Lemma~\ref{lem:sym_dual_Walsh} in this order for each of the five equalities. Finally, for the last term of (\ref{eq:worst-case_error}), we have
  \begin{align*}
 & \quad \frac{1}{|\Pcal^{\sym}|^2}\sum_{\bsz,\bsw\in \Pcal^{\sym}}K(\pi(\bsz),\pi(\bsw)) \\
 & = \frac{1}{|\Pcal^{\sym}|^2}\sum_{\bsz,\bsw\in \Pcal^{\sym}}\sum_{\bsk,\bsl\in \NN_0^s}\hat{K}(\bsk,\bsl)W_{\bsk}(\bsz)\overline{W_{\bsl}(\bsw)} \\
 & = \sum_{\bsk,\bsl\in \NN_0^s}\hat{K}(\bsk,\bsl)\frac{1}{|\Pcal^{\sym}|}\sum_{\bsz\in \Pcal^{\sym}}W_{\bsk}(\bsz)\overline{\frac{1}{|\Pcal^{\sym}|}\sum_{\bsw\in \Pcal^{\sym}}W_{\bsl}(\bsw)} \\
 & = \sum_{\bsk,\bsl\in \Pcal^{\perp}\cap \Ecal^s}\hat{K}(\bsk,\bsl) ,
  \end{align*}
where we use Equation~(\ref{eq:kernel_walsh_conv}) and Lemma~\ref{lem:sym_dual_Walsh} in the first and third equalities, respectively. Substituting these results into the right-hand side of (\ref{eq:worst-case_error}), the result follows.
\end{proof}

\begin{remark}\label{rem:sym_vs_fold}
The result of Theorem~\ref{thm:worst_case_sym_digital_net} has some similarity to that of \cite[Proposition~22]{GSY2}. When a QMC integration over a folded digital net by means of the $b$-adic tent transformation is considered, the squared worst-case error in a RKHS is given by
  \begin{align*}
     \sum_{\bsk,\bsl\in \Pcal^{\perp}\cap \Ecal^s\setminus \{\bszero\}}\hat{K}(\lfloor \bsk/b\rfloor, \lfloor \bsl/b\rfloor ) ,
  \end{align*}
where we write $\lfloor \bsx\rfloor = (\lfloor x_1\rfloor, \ldots,\lfloor x_s\rfloor)$ for $\bsx=(x_1,\ldots,x_s)\in \RR^s$. Since the number of points of a folded digital net is the same as that of an original digital net, a folded digital net has a cost advantage over a symmetrized digital net.
\end{remark}

\subsection{The mean square worst-case error}\label{sec:ms-worst}
Here we study the mean square worst-case error of symmetrized point sets with respect to a random digital shift. Now the error criterion is given by
  \begin{align*}
     \tilde{e}^2(\pi(\Pcal^{\sym}), K) = \int_{[0,1]^s}e^2(\pi(\Pcal^{\sym})\oplus \bssigma, K) \rd \bssigma ,
  \end{align*}
where the operator $\oplus$ is defined for $\bsx,\bsy\in [0,1]^s$ as
  \begin{align*}
     \bsx \oplus \bsy := \pi\left[ \sigma(\bsx)\oplus \sigma(\bsy) \right] .
  \end{align*}
From \cite[Theorem~12.4]{DPbook}, we have the following.
\begin{lemma}\label{lem:ds_kernel}
For a point set $\Pcal\subset G^s$ and a reproducing kernel $K\in L_2([0,1]^{2s})$, the mean square worst-case error of a set $\pi(\Pcal)$ with respect to a random digital shift is given by
  \begin{align*}
     \tilde{e}^2(\pi(\Pcal), K) = e^2(\pi(\Pcal), K_{\ds}) ,
  \end{align*}
where $K_{\ds}$ is called a digital shift invariant kernel defined as
  \begin{align*}
     K_{\ds}(\bsx,\bsy) := \int_{[0,1]^s}K(\bsx\oplus \bssigma, \bsy\oplus \bssigma) \rd \bssigma,
  \end{align*}
for any $\bsx,\bsy\in [0,1]^s$.
\end{lemma}

Similarly to Lemma~\ref{lem:kernel_walsh_conv}, we have the following pointwise absolute convergence of the Walsh series of $K_{\ds}$.
\begin{lemma}\label{lem:ds_kernel_walsh_conv}
Let $K$ be a continuous reproducing kernel. We assume $$\sum_{\bsk\in \NN_0^s}|\hat{K}(\bsk,\bsk)|<\infty.$$
For any $\bsz,\bsw\in G^s$, we have
\begin{align}\label{eq:ds_kernel_walsh_conv}
    K_{\ds}(\pi(\bsz), \pi(\bsw)) = \sum_{\bsk\in \NN_0^s}\hat{K}(\bsk,\bsk)W_{\bsk}(\bsz)\overline{W_{\bsk}(\bsw)}.
\end{align}
\end{lemma}

\begin{proof}
By the assumption $\sum_{\bsk\in \NN_0^s}|\hat{K}(\bsk,\bsk)|<\infty$, the right-hand side on (\ref{eq:ds_kernel_walsh_conv}) converges absolutely. Thus it suffices to show that
\begin{align*}
    \lim_{n\to \infty} \sum_{\substack{\bsk\in \NN_0^s \\ k_j<b^n, \forall j}}\hat{K}(\bsk,\bsk)W_{\bsk}(\bsz)\overline{W_{\bsk}(\bsw)} = K_{\ds}(\pi(\bsz), \pi(\bsw)).
\end{align*}
In fact we have
\begin{align*}
    & \quad \sum_{\substack{\bsk\in \NN_0^s \\ k_j<b^n, \forall j}}\hat{K}(\bsk,\bsk)W_{\bsk}(\bsz)\overline{W_{\bsk}(\bsw)} \\
    & = \sum_{\substack{\bsk\in \NN_0^s \\ k_j<b^n, \forall j}} W_{\bsk}(\bsz)\overline{W_{\bsk}(\bsw)}  \int_{[0,1]^{2s}}K(\bsx,\bsy)\overline{\wal_{\bsk}(\bsx)}\wal_{\bsk}(\bsy)\rd \bsx \rd \bsy \\
    & = \sum_{\substack{\bsk\in \NN_0^s \\ k_j<b^n, \forall j}} W_{\bsk}(\bsz)\overline{W_{\bsk}(\bsw)}  \int_{G^{2s}}K(\pi(\bsz'),\pi(\bsw'))\overline{W_{\bsk}(\bsz')}W_{\bsk}(\bsw')\rd \bsnu(\bsz') \rd \bsnu(\bsw') \\
    & = \int_{G^{2s}}K(\pi(\bsz'),\pi(\bsw'))\sum_{\substack{\bsk\in \NN_0^s \\ k_j<b^n, \forall j}} W_{\bsk}((\bsz\ominus \bsz')\ominus (\bsw\ominus \bsw')) \rd \bsnu(\bsz') \rd \bsnu(\bsw') \\
    & = \int_{G^{2s}}K(\pi(\bsz\oplus \bsz'),\pi(\bsw\oplus \bsw'))\sum_{\substack{\bsk\in \NN_0^s \\ k_j<b^n, \forall j}} W_{\bsk}(\bsw'\ominus \bsz') \rd \bsnu(\bsz') \rd \bsnu(\bsw') ,
\end{align*}
where we use Item~4 of Proposition~\ref{prop:inf_prod} twice in the second equality. Similarly to the proof of \cite[Proposition~21]{GSY2}, let us define two sets $H_n:=\{z=(\zeta_1, \zeta_2, \dots) \in G: \zeta_1=\zeta_2=\cdots =\zeta_n=0 \}$ and $J_{n,2s}:=\{(\bsz,\bsw)\in G^{2s}: \bsw\ominus \bsz\in H_n^s\}$. Then from Item~5 of Proposition~\ref{prop:inf_prod}, we have
\begin{align*}
    \sum_{\substack{\bsk\in \NN_0^s \\ k_j<b^n, \forall j}} W_{\bsk}(\bsw'\ominus \bsz') = \begin{cases}
    b^{sn} & \text{if $(\bsz',\bsw')\in J_{n,2s}$,} \\
    0      & \text{otherwise.} \\
    \end{cases}
\end{align*}
Thus we have
\begin{align*}
    & \quad \sum_{\substack{\bsk\in \NN_0^s \\ k_j<b^n, \forall j}}\hat{K}(\bsk,\bsk)W_{\bsk}(\bsz)\overline{W_{\bsk}(\bsw)} \\
    & = b^{sn}\int_{J_{n,2s}}K(\pi(\bsz\oplus \bsz'),\pi(\bsw\oplus \bsw')) \rd \bsnu(\bsz') \rd \bsnu(\bsw') \\
    & = \int_{G^s}b^{sn}\int_{\bsw'\ominus H_n^s}K(\pi(\bsz\oplus \bsz'),\pi(\bsw\oplus \bsw')) \rd \bsnu(\bsz') \rd \bsnu(\bsw') \\
    & \to \int_{G^s}K(\pi(\bsz\oplus \bsw'),\pi(\bsw\oplus \bsw'))\rd \bsnu(\bsw') \quad \text{as $n\to \infty$,}
\end{align*}
where the last convergence stems from the facts that $K\circ \pi$ is continuous since both $K$ and $\pi$ are continuous and that the product measure of the set $\bsw'\ominus H_n^s$ equals $b^{sn}$ for any $\bsw'\in G^s$. Finally we have
\begin{align*}
    & \quad \lim_{n\to \infty} \sum_{\substack{\bsk\in \NN_0^s \\ k_j<b^n, \forall j}}\hat{K}(\bsk,\bsk)W_{\bsk}(\bsz)\overline{W_{\bsk}(\bsw)} \\
    & = \int_{G^s}K(\pi(\bsz\oplus \bsw'),\pi(\bsw\oplus \bsw'))\rd \bsnu(\bsw') \\
    & = \int_{G^s}K(\pi(\bsz)\oplus \pi(\bsw'),\pi(\bsw)\oplus \pi(\bsw'))\rd \bsnu(\bsw') \\
    & = \int_{[0,1]^s}K(\pi(\bsz)\oplus \bssigma,\pi(\bsw)\oplus \bssigma)\rd \bssigma = K_{\ds}(\pi(\bsz), \pi(\bsw)),
\end{align*}
where we use Item~3 of Proposition~\ref{prop:inf_prod} in the third equality. Thus the result follows.
\end{proof}

Under some assumptions on $K$, the mean square worst-case error with respect to a random digital shift is given as follows.
\begin{theorem}\label{thm:ms_worst_case_sym_digital_net}
Let $\Pcal, \Pcal^{\perp}$ be a digital net in $G^s$ and its dual net, respectively, and $\Pcal^{\sym}$ be the symmetrized point set of $\Pcal$. Let $K$ be a continuous reproducing kernel which satisfies $$\int_{[0,1]^s}\sqrt{K(\bsx,\bsx)}\rd \bsx<\infty \quad \text{and}\quad \sum_{\bsk\in \NN_0^s}|\hat{K}(\bsk,\bsk)|<\infty.$$ The mean square worst-case error of a QMC integration over $\pi(\Pcal^{\sym})$ with respect to a random digital shift is given by
  \begin{align*}
     \tilde{e}^2(\pi(\Pcal^{\sym}), K) = \sum_{\bsk\in \Pcal^{\perp}\cap \Ecal^s\setminus \{\bszero\}}\hat{K}(\bsk,\bsk) .
  \end{align*}
\end{theorem}

\begin{proof}
We evaluate the three terms of (\ref{eq:worst-case_error}) separately in which $\pi(\Pcal^{\sym})$ and $K_{\ds}$ are substituted into $P$ and $K$, respectively. The first term of (\ref{eq:worst-case_error}) is given by
  \begin{align*}
     \int_{[0,1]^{2s}}K_{\ds}(\bsx,\bsy)\rd \bsx \rd \bsy & = \int_{G^{2s}}K_{\ds}(\pi(\bsz),\pi(\bsw))\rd \bsnu(\bsz) \rd \bsnu(\bsw) \\
     & = \sum_{\bsk\in \NN_0^s}\hat{K}(\bsk,\bsk)\int_{G^{2s}} W_{\bsk}(\bsz)\overline{W_{\bsk}(\bsw)}\rd \bsnu(\bsz) \rd \bsnu(\bsw) \\
     & = \hat{K}(\bszero,\bszero),
  \end{align*}
where we use Item~1 of Proposition~\ref{prop:inf_prod} in the third equality. For the second term of (\ref{eq:worst-case_error}), we have
\begin{align*}
     \frac{2}{|\Pcal^{\sym}|}\sum_{\bsz\in \Pcal^{\sym}}\int_{[0,1]^s}K_{\ds}(\pi(\bsz),\bsy)\rd \bsy & = \frac{2}{|\Pcal^{\sym}|}\sum_{\bsz\in \Pcal^{\sym}}\int_{G^s}K_{\ds}(\pi(\bsz),\pi(\bsw))\rd \bsnu(\bsw) \\
     & = \frac{2}{|\Pcal^{\sym}|}\sum_{\bsz\in \Pcal^{\sym}}\sum_{\bsk\in \NN_0^s}\hat{K}(\bsk,\bsk)W_{\bsk}(\bsz)\int_{G^s}\overline{W_{\bsk}(\bsw)}\rd \bsnu(\bsw) \\
     & = 2\hat{K}(\bszero,\bszero),
  \end{align*}
where we use Item~4 of Proposition~\ref{prop:inf_prod}, Equation~(\ref{eq:ds_kernel_walsh_conv}) and Item~1 of Proposition~\ref{prop:inf_prod} in this order for each of three equalities. Finally, for the last term of (\ref{eq:worst-case_error}), we have
  \begin{align*}
 & \quad \frac{1}{|\Pcal^{\sym}|^2}\sum_{\bsz,\bsw\in \Pcal^{\sym}}K_{\ds}(\pi(\bsz),\pi(\bsw)) \\
 & = \frac{1}{|\Pcal^{\sym}|^2}\sum_{\bsz,\bsw\in \Pcal^{\sym}}\sum_{\bsk\in \NN_0^s}\hat{K}(\bsk,\bsk)W_{\bsk}(\bsz)\overline{W_{\bsk}(\bsw)} \\
 & = \sum_{\bsk\in \NN_0^s}\hat{K}(\bsk,\bsk)\frac{1}{|\Pcal^{\sym}|}\sum_{\bsz\in \Pcal^{\sym}}W_{\bsk}(\bsz)\overline{\frac{1}{|\Pcal^{\sym}|}\sum_{\bsw\in \Pcal^{\sym}}W_{\bsk}(\bsw)} \\
 & = \sum_{\bsk\in \Pcal^{\perp}\cap \Ecal^s}\hat{K}(\bsk,\bsk) ,
  \end{align*}
where we use Equation~(\ref{eq:ds_kernel_walsh_conv}) and Lemma~\ref{lem:sym_dual_Walsh} in the first and third equalities, respectively. Substituting these results into the right-hand side of (\ref{eq:worst-case_error}), the result follows.
\end{proof}

\begin{remark}
The result of Theorem~\ref{thm:ms_worst_case_sym_digital_net} has some similarity to that of \cite[Theorem~16]{GSY1}. When a QMC integration over a ``digitally shifted and then folded'' digital net is considered, the mean square worst-case error with respect to a random digital shift in a RKHS is given by
  \begin{align*}
     \sum_{\bsk\in \Pcal^{\perp}\cap \Ecal^s\setminus \{\bszero\}}\hat{K}(\lfloor \bsk/b\rfloor, \lfloor \bsk/b\rfloor ) ,
  \end{align*}
where again we write $\lfloor \bsx\rfloor = (\lfloor x_1\rfloor, \ldots,\lfloor x_s\rfloor)$ for $\bsx=(x_1,\ldots,x_s)\in \RR^s$. As already mentioned in Remark~\ref{rem:sym_vs_fold}, this approach has a cost advantage over a symmetrized digital net.
\end{remark}

\begin{remark}
Following the same arguments as in \cite[Sections~4--6]{GSY2}, Theorems~\ref{thm:worst_case_sym_digital_net} and \ref{thm:ms_worst_case_sym_digital_net} can be applied to component-by-component (CBC) construction or Korobov construction of good symmetrized (higher order) polynomial lattice rules which achieve high order convergence of the worst-case error in an unanchored Sobolev space of smoothness $\alpha\in \NN$, $\alpha\ge 2$. For instance, for $m\in \NN$ and a prime $b$, the CBC construction can find symmetrized higher order polynomial lattice rules with $b^{m+s}$ points which achieve the worst-case error convergence of $O(b^{-\alpha m+\epsilon})$ with an arbitrary small $\epsilon>0$. Moreover, this construction can be done in $O(s\alpha m b^{\alpha m/2})$ arithmetic operations using $O(b^{\alpha m/2})$ memory.
\end{remark}

\section{Discrepancy bounds for symmetrized Hammersley point sets}\label{sec:ham}
Finally in this paper we prove that symmetrized Hammersley point sets in prime base $b$ achieve the best possible order of $L_p$ discrepancy for all $1\le p< \infty$. According to Definitions~\ref{def:hammersley} and \ref{def:b_sym}, the two-dimensional symmetrized Hammersley point set in base $b$ is given as follows.
\begin{definition}\label{def:sym_hammersley}
Let $b\ge 2$ be a positive integer. For $m\in \NN$, the two-dimensional symmetrized Hammersley point set in base $b$ is a point set consisting of $b^{m+2}$ defined as
  \begin{align*}
    P_H^{\sym} := \left\{ (x_{\bsa},y_{\bsa}): \bsa=(a_1,\ldots,a_{m+2})\in \{0,1,\ldots,b-1\}^{m+2} \right\},
  \end{align*}
where
  \begin{align*}\begin{cases}
    x_{\bsa} = \frac{a_1\oplus a_{m+1}}{b}+\cdots + \frac{a_m\oplus a_{m+1}}{b^m} + \frac{a_{m+1}}{b^{m+1}}+ \frac{a_{m+1}}{b^{m+2}}+\cdots , \\
    y_{\bsa} = \frac{a_m\oplus a_{m+2}}{b}+\cdots + \frac{a_1\oplus a_{m+2}}{b^m} + \frac{a_{m+2}}{b^{m+1}}+ \frac{a_{m+2}}{b^{m+2}}+\cdots .
    \end{cases}
  \end{align*}
\end{definition}

We also consider a truncated version of $P_H^{\sym}$. For a positive integer $n\ge m+2$, the two-dimensional \emph{truncated} symmetrized Hammersley point set in base $b$, denoted by $P_H^{\sym, n}$, is given as
  \begin{align*}
    P_H^{\sym, n} := \left\{ (x^n_{\bsa},y^n_{\bsa}): \bsa=(a_1,\ldots,a_{m+2})\in \{0,1,\ldots,b-1\}^{m+2} \right\} ,
  \end{align*}
where
  \begin{align*}\begin{cases}
    x^n_{\bsa} = \frac{a_1\oplus a_{m+1}}{b}+\cdots + \frac{a_m\oplus a_{m+1}}{b^m} + \frac{a_{m+1}}{b^{m+1}}+ \cdots + \frac{a_{m+1}}{b^n} , \\
    y^n_{\bsa} = \frac{a_m\oplus a_{m+2}}{b}+\cdots + \frac{a_1\oplus a_{m+2}}{b^m} + \frac{a_{m+2}}{b^{m+1}}+ \cdots + \frac{a_{m+2}}{b^n} .
    \end{cases}
  \end{align*}
Here note that $P_H^{\sym, n}$ is a digital net over $\ZZ_b$ with generating matrices of size $n \times (m+2)$
  \begin{align}\label{eq:matrix_sym_ham}
    C_1 = \left(
    \begin{array}{cccccc}
      1 & 0 & \cdots & 0 & 1 & 0 \\
      0 & 1 & \cdots & 0 & 1 & 0 \\
      \vdots & \vdots & \ddots & \vdots & \vdots & \vdots \\
      0 & 0 & \cdots & 1 & 1 & 0 \\
      0 & 0 & \cdots & 0 & 1 & 0 \\
      \vdots & \vdots & \ddots & \vdots & \vdots & \vdots \\
      0 & 0 & \cdots & 0 & 1 & 0 \\
    \end{array} \right) ,
    C_2 = \left(
    \begin{array}{cccccc}
      0 & \cdots & 0 & 1 & 0 & 1 \\
      0 & \cdots & 1 & 0 & 0 & 1 \\
      \vdots & \text{\reflectbox{$\ddots$}} & \vdots & \vdots & \vdots & \vdots \\
      1 & \cdots & 0 & 0 & 0 & 1 \\
      0 & \cdots & 0 & 0 & 0 & 1 \\
      \vdots & \text{\reflectbox{$\ddots$}} & \vdots & \vdots & \vdots & \vdots \\
      0 & \cdots & 0 & 0 & 0 & 1 \\
    \end{array} \right) .
  \end{align}
The following theorem is the main result of this section.
\begin{theorem}\label{thm:disc_sym_ham}
For a prime $b$ and $m\in \NN$, let $P_H^{\sym}$ be the symmetrized Hammersley point set in base $b$ consisting of $N=b^{m+2}$ points. Then, for all $1\le p< \infty$, the $L_p$ discrepancy of $P_H^{\sym}$ is of the best possible order. Namely, for any $1\le p<\infty$ there exists a constant $C_p>0$ such that
  \begin{align*}
    L_p(P_H^{\sym}) \le C_p \frac{\sqrt{m+2}}{b^{m+2}} = C_p \frac{\sqrt{\log_b N}}{N}.
  \end{align*}
\end{theorem}

Our proof consists of two parts. We first prove that the $L_p$ discrepancy of $P_H^{\sym, n}$ is of the best possible order for any $1\le p< \infty$ when $n> 2m$. Then we show that the difference between the $L_p$ discrepancies of $P_H^{\sym}$ and $P_H^{\sym,n}$ is small enough that the $L_p$ discrepancy of $P_H^{\sym}$ is still of the best possible order for any $1\le p< \infty$.

\subsection{The minimum Dick weight of point sets}
To prove the first part, we introduce the Dick weight function $\mu_2$ given in \cite{Dick08,Dick14} and the minimum Dick weight $\rho_2(P)$ for a two-dimensional digital net $P$.
\begin{definition}\label{def:dick_weight}
For $k\in \NN$, we denote its $b$-adic expansion by $k=\kappa_1b^{a_1-1}+\kappa_2b^{a_2-1}+\dots+\kappa_vb^{a_v-1}$ such that $\kappa_1,\dots,\kappa_v\in \{1,\dots,b-1\}$ and $a_1>a_2>\dots > a_v>0$. Then the Dick weight function $\mu_2:\NN_0\to \RR$ is defined as
  \begin{align*}
    \mu_2(k) = \begin{cases}
    a_1+a_2 & \text{if $v\ge 2$}, \\
    a_1     & \text{if $v=1$}, \\
    0       & \text{if $k=0$}.
    \end{cases}
  \end{align*}
For vectors $\bsk=(k_1,k_2)\in \NN_0^2$, we define $\mu_2(\bsk):=\mu_2(k_1)+\mu_2(k_2)$. Moreover, let $P$ be a two-dimensional digital net over $\ZZ_b$. Then the minimum Dick weight $\rho_2(P)$ is defined as
  \begin{align*}
    \rho_2(P) := \min_{\bsk\in P^{\perp}\setminus \{(0,0)\}}\mu_2(\bsk) .
  \end{align*}
\end{definition}

The following lemma shows how the minimum Dick weight of a digital net connects with a structure of its generating matrices, see \cite[Chapter~15]{DPbook}.
\begin{lemma}\label{lem:dick}
For $m,n\in \NN$ with $n\ge m$, let $P$ be a digital net over $\ZZ_b$ with generating matrices $C_1,C_2$ of size $n\times m$. For $j=1,2$ and $1\le l\le n$, let $\bsc_{j,l}$ denote the $l$-th row vector of $C_j$. When $l>n$, $\bsc_{j,l}$ denotes the vector consisting of $m$ zeros. Let $\rho$ be a positive integer such that for all $1\le i_{1,v_1}<\dots<i_{1,1}\le 2m$ and $1\le i_{2,v_2}<\dots<i_{2,1}\le 2m$ with $v_j\in \NN_0$ and
  \begin{align*}
    \sum_{l=1}^{\min(v_1,2)}i_{1,l}+\sum_{l=1}^{\min(v_2,2)}i_{2,l}\le \rho ,
  \end{align*}
the vectors $\bsc_{1,i_{1,v_1}},\dots,\bsc_{1,i_{1,1}},\bsc_{2,i_{2,v_2}},\dots,\bsc_{2,i_{2,1}}$ are linearly independent over $\ZZ_b$. Then we have $\rho_2(P)>\rho$.
\end{lemma}

Although Definition~\ref{def:dick_weight} and Lemma~\ref{lem:dick} are considered only for the two-dimensional case, they can be generalized to the $s$-dimensional case with any $s\in \NN$. Recently, Dick \cite{Dick14} proved that digital nets over $\ZZ_2$ with large minimum Dick weight achieve the best possible order of the $L_p$ discrepancy for $1<p<\infty$ and for any number of dimensions. For the two-dimensional case, his result also implies the best possible order of the $L_1$ discrepancy with respect to the general lower bound by Hal\'{a}sz \cite{Hal81}. Dick's result was generalized more recently in \cite{Mark15} to digital nets over $\ZZ_b$ for a prime $b$. We specialize their results (\cite[Corollary~2.2]{Dick14} and \cite[Corollary~1.8]{Mark15}) on the $L_p$ discrepancy of digital nets for the two-dimensional case.

\begin{proposition}\label{prop:Dick-Mark}
Let $P$ be a digital net over $\ZZ_b$ consisting of $b^m$ points which satisfies $\rho_2(P)> 2m-t$ for some integer $0\le t\le 2m$. Then for all $1\le p<\infty$ there exists a constant $C_{p,t}$ which depends only on $t$ and $p$ such that we have
  \begin{align*}
    L_p(P) \le C_{p,t} \frac{\sqrt{m}}{b^m}.
  \end{align*}
\end{proposition}

From Lemma~\ref{lem:dick} and Proposition~\ref{prop:Dick-Mark}, as the first part of the proof of Theorem~\ref{thm:disc_sym_ham}, it suffices to show the linear independence properties of the generating matrices (\ref{eq:matrix_sym_ham}) such that the minimum Dick weight of $P_H^{\sym, n}$ is large.

\begin{lemma}\label{lem:linear_ independence}
For $m,n\in \NN$ with $n> 2m$, let $C_1$ and $C_2$ denote the generating matrices (\ref{eq:matrix_sym_ham}). For $j=1,2$ and $1\le l\le n$, $\bsc_{j,l}$ denotes the $l$-th row vector of $C_j$. The following sets of the vectors are linearly independent over $\ZZ_b$:
\begin{enumerate}
\item $\{\bsc_{1,1},\dots,\bsc_{1,r},\bsc_{2,1},\dots,\bsc_{2,m+1-r}\}$ for $0\le r\le m+1$,
\item $\{\bsc_{1,1},\dots,\bsc_{1,m+1},\bsc_{2,r}\}$ and $\{\bsc_{2,1},\dots,\bsc_{2,m+1},\bsc_{1,r}\}$ for $1\le r\le m$,
\item $\{\bsc_{1,1},\dots,\bsc_{1,r},\bsc_{2,1},\dots,\bsc_{2,m-r},\bsc_{j,s}\}$ for $j=1,2$, $0\le r\le m$, and $m+1\le s\le n$,
\item $\{\bsc_{1,1},\dots,\bsc_{1,r_{1,2}},\bsc_{1,r_{1,1}},\bsc_{2,1},\dots,\bsc_{2,r_{2,2}},\bsc_{2,r_{2,1}}\}$ for $0< r_{1,2}<r_{1,1}\le m$ and $0< r_{2,2}<r_{2,1}\le m$ such that $r_{1,1}+r_{1,2}+r_{2,1}+r_{2,2}\le 2m+1$.
\end{enumerate}
\end{lemma}
\noindent Since the proof is quite similar to that of \cite[Lemma~3.1]{Goda14}, we omit it. Using the linear independence properties of (\ref{eq:matrix_sym_ham}) shown in the above lemma together with Lemma~\ref{lem:dick} and Proposition~\ref{prop:Dick-Mark}, we have the following.
\begin{proposition}\label{prop:disc_trun_sym_ham}
For $m,n\in \NN$ with $n> 2m$, let $P_H^{\sym,n}$ be the truncated symmetrized Hammersley point set in base $b$ consisting of $N=b^{m+2}$ points. The minimum Dick weight of $P_H^{\sym,n}$ is larger than $2m+1$, that is,
  \begin{align*}
    \rho_2(P_H^{\sym,n}) > 2m+1.
  \end{align*}
This implies that the $L_p$ discrepancy of $P_H^{\sym,n}$ is bounded by
  \begin{align*}
    L_p(P_H^{\sym,n}) \le C_{p} \frac{\sqrt{m+2}}{b^{m+2}}= C_p \frac{\sqrt{\log_b N}}{N},
  \end{align*}
for any $p\in [1,\infty)$ with a positive constant $C_p$ depending only on $p$.
\end{proposition}
\noindent Since the proof on the minimum Dick weight of $P_H^{\sym,n}$ is quite similar to that of \cite[Lemma~3.3]{Goda14}, we omit it and just give some comments on the $L_p$ discrepancy bound instead. As $P_H^{\sym,n}$ consists of $b^{m+2}$ points and the minimum Dick weight of $P_H^{\sym,n}$ is larger than $2m+1=2(m+2)-3$, the $t$-value in Proposition~\ref{prop:Dick-Mark} of $P_H^{\sym,n}$ always equals 3 for any $m\in \NN$. This is why we simply write $C_p$ instead of $C_{p,t}$ in the above proposition. This $t$-value is same as that for two-dimensional folded Hammersley point sets as given in \cite[Lemma~3.3]{Goda14}. Moreover, if one wants to get an explicit value of $C_p$, it might be better to use the Littlewood-Paley inequality in conjunction with the Haar coefficients of the local discrepancy function for $P_H^{\sym,n}$ as done in \cite{HKP15}, instead of our proof based on the linear independence properties. However, this is beyond the scope of this paper.

\subsection{Effect of the truncation on the $L_p$ discrepancy}
As the second part of the proof of Theorem~\ref{thm:disc_sym_ham}, we show that the difference between the $L_p$ discrepancies of $P_H^{\sym}$ and $P_H^{\sym,n}$ is small enough that the $L_p$ discrepancy of $P_H^{\sym}$ is still of the best possible order for any $1\le p< \infty$. Namely we show the following.
\begin{proposition}\label{prop:effect_trun}
For $m,n\in \NN$ with $n\ge m+2$, let $P_H^{\sym}$ be the symmetrized Hammersley point set in base $b$ consisting of $N=b^{m+2}$ points and $P_H^{\sym,n}$ be its truncated point set. Then we have
  \begin{align}\label{eq:effect_trun}
    L_p(P_H^{\sym}) \le L_p(P_H^{\sym,n}) + \frac{1}{b^{m+2(n-m-1)/p}} ,
  \end{align}
for any $p\in [1,\infty)$.
\end{proposition}
\noindent Before providing the proof of Proposition~\ref{prop:effect_trun}, it should be mentioned that we arrive at the result of Theorem~\ref{thm:disc_sym_ham} by combining Propositions~\ref{prop:disc_trun_sym_ham} and \ref{prop:effect_trun} since the second term on the right-hand side of (\ref{eq:effect_trun}) is small enough for any $p\in [1,\infty)$ that it does not affect the order of the $L_p$ discrepancy when $n> 2m$.

\begin{proof} From the definition of the $L_p$ discrepancy we have
  \begin{align*}
    L_p(P_H^{\sym}) & = \left( \int_{[0,1]^2}|\Delta(\bst;P_H^{\sym})|^p \rd \bst \right)^{1/p} \\
    & = \left( \int_{[0,1]^2}|\Delta(\bst;P_H^{\sym})-\Delta(\bst;P_H^{\sym,n})+\Delta(\bst;P_H^{\sym,n})|^p \rd \bst \right)^{1/p} \\
    & \le L_p(P_H^{\sym,n}) + \left( \int_{[0,1]^2}|\Delta(\bst;P_H^{\sym})-\Delta(\bst;P_H^{\sym,n})|^p \rd \bst \right)^{1/p} ,
  \end{align*}
where we use the Minkowski inequality. Thus, we shall focus on the second term on the rightmost-hand side in the following.

Now for any $\bst\in [0,1]^2$ we have
  \begin{align*}
    \Delta(\bst;P_H^{\sym})-\Delta(\bst;P_H^{\sym,n}) = \frac{1}{b^{m+2}}\sum_{\bsa\in \{0,1,\ldots,b-1\}^{m+2}}\left( 1_{[\bszero,\bst)}(x_{\bsa},y_{\bsa})-1_{[\bszero,\bst)}(x^n_{\bsa},y^n_{\bsa})\right) .
  \end{align*}
For the summand on the right-hand side, we have 
  \begin{align*}
    1_{[\bszero,\bst)}(x_{\bsa},y_{\bsa})-1_{[\bszero,\bst)}(x^n_{\bsa},y^n_{\bsa}) = \begin{cases}
    1  & \text{if $(x_{\bsa},y_{\bsa})\in [\bszero,\bst)$ and $(x^n_{\bsa},y^n_{\bsa})\notin [\bszero,\bst)$,} \\
    -1 & \text{if $(x_{\bsa},y_{\bsa})\notin [\bszero,\bst)$ and $(x^n_{\bsa},y^n_{\bsa})\in [\bszero,\bst)$,} \\
    0 & \text{otherwise.}
    \end{cases}
  \end{align*}
For a given set $\bsa=\{a_1,\ldots,a_{m+2}\}$, we have
  \begin{align*}
    x^n_{\bsa}\le x_{\bsa} & = \frac{a_1\oplus a_{m+1}}{b}+\cdots + \frac{a_m\oplus a_{m+1}}{b^m} + \frac{a_{m+1}}{b^{m+1}}+ \cdots + \frac{a_{m+1}}{b^n} +  \frac{a_{m+1}}{b^{n+1}}+\cdots \\
    & = x^n_{\bsa}+\frac{a_{m+1}}{b^n(b-1)}\le x^n_{\bsa}+\frac{1}{b^n} .
  \end{align*}
Similarly we have
  \begin{align*}
    y^n_{\bsa}\le y_{\bsa} = y^n_{\bsa}+\frac{a_{m+2}}{b^n(b-1)} \le y^n_{\bsa}+\frac{1}{b^n} .
  \end{align*}
This implies that we never have the case where $(x_{\bsa},y_{\bsa})\in [\bszero,\bst)$ and $(x^n_{\bsa},y^n_{\bsa})\notin [\bszero,\bst)$. Thus, for any $\bst\in [0,1]^2$ we have
  \begin{align*}
    & \quad |\Delta(\bst;P_H^{\sym})-\Delta(\bst;P_H^{\sym,n})| \\
    & = \frac{|\left\{ \bsa\in \{0,1,\ldots,b-1\}^{m+2}: (x_{\bsa},y_{\bsa})\notin [\bszero,\bst)\; \text{and}\; (x^n_{\bsa},y^n_{\bsa})\in [\bszero,\bst) \right\}|}{b^{m+2}} \\
    & = \frac{|\left\{ \bsa\in \{0,1,\ldots,b-1\}^{m+2}: x^n_{\bsa}<t_1\le x_{\bsa}\; \text{and}\; y^n_{\bsa}<t_2\le y_{\bsa} \right\}|}{b^{m+2}} \\
    & \le \frac{|\left\{ \bsa\in \{0,1,\ldots,b-1\}^{m+2}: x^n_{\bsa}<t_1\le x^n_{\bsa}+1/b^n\; \text{and}\; y^n_{\bsa}<t_2\le y^n_{\bsa}+1/b^n \right\}|}{b^{m+2}}.
  \end{align*}
From the definition of $P_H^{\sym,n}$, the first coordinate $x^n_{\bsa}$ takes the value of the form
  \begin{align*}
    \frac{c_1}{b^m}+\frac{c_2}{b-1}\left( \frac{1}{b^m}-\frac{1}{b^n}\right)
  \end{align*}
for $c_1,c_2\in \NN_0$ with $0\le c_1<b^m$ and $0\le c_2<b$. Since $x^n_{\bsa}$ does not depend on $a_{m+2}$, each value repeats $b$ times. The length of the interval between two distinct consecutive elements of $x^n_{\bsa}$ equals either
  \begin{align*}
    \frac{1}{b-1}\left( \frac{1}{b^m}-\frac{1}{b^n}\right) \quad \text{or}\quad \frac{1}{b^n}.
  \end{align*}
Note that the former length is greater than $1/b^n$ when $n\ge m+2$.
This also holds true for the second coordinate $y^n_{\bsa}$. Therefore, when there exist $c_1,c_2,d_1,d_2\in \NN_0$ with $0\le c_1,d_1<b^m$ and $0\le c_2,d_2<b$ such that
  \begin{align*}
    \frac{c_1}{b^m}+\frac{c_2}{b-1}\left( \frac{1}{b^m}-\frac{1}{b^n}\right) < t_1 \le \frac{c_1}{b^m}+\frac{c_2}{b-1}\left( \frac{1}{b^m}-\frac{1}{b^n}\right)+\frac{1}{b^n}
  \end{align*}
and
  \begin{align*}
    \frac{d_1}{b^m}+\frac{d_2}{b-1}\left( \frac{1}{b^m}-\frac{1}{b^n}\right) < t_2 \le \frac{d_1}{b^m}+\frac{d_2}{b-1}\left( \frac{1}{b^m}-\frac{1}{b^n}\right)+\frac{1}{b^n} ,
  \end{align*}
there are at most $b^2$ points of $P_H^{\sym,n}$ which satisfy $x^n_{\bsa}<t_1\le x^n_{\bsa}+1/b^n$ and $y^n_{\bsa}<t_2\le y^n_{\bsa}+1/b^n$. Note that the Lebesgue measure of the set consisting of such $\bst$ is given by $1/b^{2(n-m-1)}$. On the other hand, when there do not exist $c_1,c_2,d_1,d_2\in \NN_0$ with $0\le c_1,d_1<b^m$ and $0\le c_2,d_2<b$ which satisfy the above condition, there is no point of $P_H^{\sym,n}$ which satisfy $x^n_{\bsa}<t_1\le x^n_{\bsa}+1/b^n$ and $y^n_{\bsa}<t_2\le y^n_{\bsa}+1/b^n$. Therefore, we have
  \begin{align*}
    \left( \int_{[0,1]^2}|\Delta(\bst;P_H^{\sym})-\Delta(\bst;P_H^{\sym,n})|^p \rd \bst \right)^{1/p} & \le \frac{b^2}{b^{m+2}}\cdot \left( \frac{1}{b^{2(n-m-1)}}\right)^{1/p} \\
    & = \frac{1}{b^{m+2(n-m-1)/p}}.
  \end{align*}
Hence the result follows.
\end{proof}


\end{document}